\documentclass[11pt]{article}

\usepackage{amsmath, amsfonts, amsthm, verbatim}

\newcommand{\R}{\mathbb{R}}
\newcommand{\N}{\mathbb{N}}
\renewcommand{\d}{\mathrm{d}}
\renewcommand{\i}{\mathfrak{i}}
\newcommand{\E}{\mathbf{E}}
\newcommand{\p}{\mathbf{P}}

\newcommand{\dint}{\int \hspace{-6pt} \int}

\theoremstyle{plain}
\newtheorem{lemma}{Lemma}
\newtheorem{theorem}{Theorem}
\newtheorem{proposition}{Proposition}
\newtheorem{corollary}{Corollary}

\title{Ergodic properties of generalized Ornstein--Uhlenbeck processes}

\begin{document}

\author{P\'eter Kevei\thanks{This research was funded by a postdoctoral fellowship
of the Alexander von Humboldt Foundation.} \\
Center for Mathematical Sciences, Technische Universit\"at M\"unchen \\
Boltzmannstra{\ss }e 3, 85748 Garching, Germany \\
\texttt{peter.kevei@tum.de}}

\maketitle

\begin{abstract}
We investigate ergodic properties of generalized Ornstein--Uhlenbeck processes. In 
particular, we provide sufficient conditions for ergodicity, and for  subexponential and  
exponential convergence to the invariant probability measure. We use the Foster--Lyapunov 
method. The drift conditions are obtained using the explicit form of the generator of the 
continuous process. In some special cases the optimality of our results can be shown.

\end{abstract}

\section{Introduction}

Let $(U_t, L_t)_{t \geq 0}$ be a bivariate L\'evy process with characteristic exponent
\begin{equation*} \label{eq:UVchf}
\begin{split}
& \log \E e^{\i (\theta_1 U_1 + \theta_2 L_1)} \\
&= 
\i (\theta_1 \gamma_U + \theta_2 \gamma_L) 
-\frac{1}{2} \langle \theta \Sigma,  \theta \rangle +
\dint_{\R^2} \left( e^{\i (\theta_1 z_1 + \theta_2 z_2)} - 1 -
\i (\theta_1 z_1 + \theta_2 z_2) I(|z| \leq 1) \right) \nu_{UL}(\d z) 
,
\end{split}
\end{equation*}
where $\gamma_U, \gamma_L \in \R$,
\[
\Sigma =
\begin{pmatrix}
\sigma_U^2 & \sigma_{UL} \\
\sigma_{UL} & \sigma_L^2
\end{pmatrix}
\in \R^{2 \times 2}
\]
is a nonnegative 
semidefinite matrix, $\nu_{UL}$ is a bivariate L\'evy measure.
Here and later on $\langle \cdot, \cdot \rangle$ stands for the usual inner product in 
$\R^2$, $| \cdot |$ is the Euclidean norm both in $\R$ and $\R^2$, and
$I(\cdot)$ stands for the indicator function.
To ease the notation we also write $\theta = (\theta_1, \theta_2)$, $z = (z_1, z_2)$. Let 
 $\nu_U, \nu_L$ denote the L\'evy measure of $U$ and $L$, respectively.
If $\dint_{|z| \leq 1} |z| \nu_{UL}(\d z) < \infty$ we write
\begin{equation*} \label{eq:UVchf2}
\E e^{\i (\theta_1 U_1 + \theta_2 L_1)} =
\exp \left\{ \i (\theta_1 \gamma'_U + \theta_2 \gamma'_L) 
-\frac{1}{2} \langle \theta \Sigma, \theta \rangle +
\dint_{\R^2} \left( e^{\i (\theta_1 z_1 + \theta_2 z_2)} - 1 \right)
\nu_{UL}(\d z)\right\},
\end{equation*}
where $(\gamma_U', \gamma_L')$ is the drift.

In the present paper we investigate ergodic properties of the 
unique solution of the stochastic differential equation
\begin{equation} \label{eq:SDEGOU}
\d V_t = V_{t-} \d U_t + \d L_t, \quad t \geq 0. 
\end{equation}
The  unique solution to (\ref{eq:SDEGOU}) was determined by Behme, Lindner, and Maller 
\cite[Proposition 3.2]{BLM}. Introduce the L\'evy process $\eta$ as
\begin{equation} \label{eq:def-eta}
\eta_t = L_t - \sum_{s\leq t} 
\frac{\Delta U_s \Delta L_s}{1 + \Delta U_s} - \sigma_{UL} t,
\end{equation}
where for any c\`adl\`ag process $Y$ its jump at $t$ is $\Delta Y_t = Y_t - Y_{t-}$.
If $\nu_U(\{ -1 \} ) = 0$ then the solution to (\ref{eq:SDEGOU}) can be written as
\begin{equation*} \label{eq:defV}
V_t = \mathcal{E}(U)_t \left[ V_ 0 + \int_0^t \mathcal{E}(U)_{s-}^{-1} \d \eta_s  
\right], 
\end{equation*}
where the stochastic exponential (Dol\'eans--Dade exponential) $\mathcal{E}(U)$ is
\[
\mathcal{E}(U)_t =  e^{U_t - \sigma_U^2 t/2} \prod_{s \leq t} (1 + \Delta U_s)
e^{-\Delta U_s}.
\]
While for $\nu_U(\{ -1 \} ) > 0$
\begin{equation} \label{eq:defV2}
\begin{split}
V_t & = \mathcal{E}(U)_t \left[ V_0 +  \int_0^t \mathcal{E}(U)_{s-}^{-1} \d \eta_s 
\right]
I(K(t) = 0) \\
& \phantom{=} \ + \mathcal{E}(U)_{(T(t),t]} \left[ \Delta L_{T(t)} +  \int_{T(t)}^t 
\mathcal{E}(U)_{(T(t),s)}^{-1} \d \eta_s \right] I(K(t) \geq 1), 
\end{split}
\end{equation}
where $K(t) = \# \{ s \in (0,t]:\, \Delta U(s) = -1 \} $, 
$T(t) = \sup \{ s \in (0,t]: \, \Delta U(s) = -1 \}$, and for $s \leq t$
\[
\begin{split}
\mathcal{E}(U)_{(s,t]} & =  e^{U_t - U_s - \sigma_U^2 (t-s)/2} 
\prod_{s< u \leq t} (1 + \Delta U_u) e^{-\Delta U_u}, \\
\mathcal{E}(U)_{(s,t)} & =  e^{U_{t-} - U_s - \sigma_U^2 (t-s)/2} 
\prod_{s< u < t} (1 + \Delta U_u) e^{-\Delta U_u}.
\end{split}
\]
Here we always consider deterministic initial value $V_0 = x_0 \in \R$, in particular it 
is independent of $(U,L)$. The processes $U$ and $L$ are semimartingales with respect to 
the smallest filtration, which satisfies the usual hypotheses and contains the filtration 
generated by $(U,L)$. Stochastic integrals are always meant with respect to this 
filtration. Since we do not directly use stochastic analysis, we prefer to suppress 
unnecessary notation.

If $\nu_U((-\infty, -1])=0$ then  we may introduce the L\'evy process $\xi$ as
\begin{equation} \label{eq:def-xi}
\xi_t = - \log \mathcal{E}(U)_t =
- U_t + \frac{\sigma_U^2}{2} t + \sum_{s \leq t} [ \Delta U_s - \log (1 + \Delta U_s) ].
\end{equation}
Then $(\xi, \eta)$ is a bivariate L\'evy process.
The generalized Ornstein--Uhlenbeck process (GOU) corresponding to $(\xi, \eta)$ is
\begin{equation} \label{eq:GOU}
V_t = e^{-\xi_t} \left[ V_0 + \int_0^t e^{\xi_{s-}} \d \eta_s  \right],
\end{equation}
where $V_0$ is independent of $(\xi, \eta)$. 
In fact there is a one-to-one correspondence between bivariate 
L\'evy processes $(\xi, \eta)$ and $(U, L)$, where $\nu_U((-\infty, -1]) = 0$.
Thus, without the restriction $\nu_U((-\infty, -1]) = 0$ the class of solutions to 
(\ref{eq:SDEGOU}) is a larger than the GOU processes.
In the present paper we deal with the general case, and we use the description through 
$(U, L)$. When $U_t = - \mu t$, $\mu > 0$, $V$ is called L\'evy-driven 
Ornstein--Uhlenbeck process, and if $L_t$ is a Brownian motion, then we obtain the 
classical Ornstein--Uhlenbeck process.

\medskip

Stationary GOU processes, or more generally stationary solutions to (\ref{eq:SDEGOU}) have
long attracted much attention in the probability community. De Haan and Karandikar 
\cite{dHK} showed that these processes are the natural continuous time analogues of 
perpetuities. Carmona, Petit, and Yor \cite{CPY} gave sufficient conditions in order that 
$V$ in (\ref{eq:GOU}) converges in distribution to the stationary distribution for any 
nonstochastic $V_0 = x_0$. Necessary and sufficient conditions for the existence of a 
stationary solution were given by Lindner and Maller \cite{LM} in the GOU case, and 
by Behme, Lindner, and Maller \cite{BLM} in case of solutions to (\ref{eq:SDEGOU}). 
Tail behavior and moments of the stationary solution was investigated by Behme \cite{B11}.
The stationary solution under appropriate conditions is 
$\int_0^\infty e^{-\xi_{s-}} \d L_s$, which is the exponential functional of the bivariate 
L\'evy process $(\xi, L)$. Continuity properties of these exponential functionals were 
investigated by Carmona, Petit, and Yor \cite{CPY2}, Bertoin, Lindner, and Maller 
\cite{BLM08}, Lindner and Sato \cite{LS}, and Kuznetsov, Pardo, and Savov \cite{KPS}.
Wiener--Hopf factorization of exponential functionals of L\'evy processes (when $L_t 
\equiv t$) was treated by Pardo, Patie, and Savov \cite{PPS}. GOU processes have a wide 
range of applications, among others in mathematical physics, in finance, and in risk 
theory. For a more complete account on GOU processes and on exponential functionals of 
L\'evy processes  we refer to the survey paper by Bertoin and Yor \cite{BY}, to Behme 
and Lindner \cite{BL15}, and to \cite{KPS}, and the references therein.

Here we deal with ergodic properties of GOU processes. Ergodicity of stochastic processes 
is important on its own right, and also in applications, such as estimation of certain 
parameters. Ergodic theory for general Markov process, both in the discrete and in the 
continuous case was developed  by Meyn and Tweedie \cite{MT1, MT2, MT3}. Using the 
so-called Foster--Lyapunov techniques, they worked out conditions for ergodicity and 
exponential ergodicity in terms of the generator of the underlying process. Recently, 
much attention is drawn to situations where the rate of convergence is only 
subexponential. Fort and Roberts \cite{FR}, Douc, Fort, and Guillin \cite{DFG} and 
Bakry, Cattiaux,  and Guillin \cite{BCG} proved general conditions for subexponential 
rates. See also the lecture notes by Hairer \cite{H}.

Concerning OU processes, Sato and Yamazato \cite{SY} gave necessary and sufficient 
conditions for the convergence of a L\'evy-driven  OU process. Exponential ergodicity  
was investigated by Masuda \cite{M} and Wang \cite{W} in the L\'evy-driven case, and by 
Fasen \cite{F} and Lee \cite{L} for GOU processes. Parameter estimation for GOU processes 
was treated by Belomestny and Panov \cite{BP}.

\smallskip

The paper is organized as follows. 
Section \ref{sect:main} contains the main results of the paper. After fixing the basic 
notation, in Theorem \ref{thm:genCD2} under general integrability assumptions we prove 
ergodicity for $V$.  In particular, 
the assumptions in Theorem \ref{thm:genCD2} reduce to the necessary and sufficient 
condition by Sato and Yamazato \cite{SY} in the L\'evy-drive OU case. In Theorems 
\ref{thm:gensubexp} and \ref{thm:gensubexp2} we obtain two different subexponential 
rates: a polynomial and an `almost exponential' one. Moreover, we point out in 
Proposition \ref{prop:CDsubexp-gen} in Section \ref{sect:proofs} that under more complex 
moment assumptions more general subexponential rates can be obtained. These results are 
particularly interesting in view of the rare subexponential convergence rates. In 
fact, in his Remark 4.4 \cite{M} Masuda claimed that in most cases stationary 
L\'evy-driven OU processes are exponentially ergodic. However, we also mention 
that subexponential rates were found in some special cases in \cite{FR, DFG}. 
Theorem \ref{thm:genCD3} provides sufficient conditions for exponential ergodicity. 
Exponential ergodicity was proved by Lee \cite{L} under stronger conditions.
In Theorems \ref{thm:genCD2}--\ref{thm:genCD3} we assume that $\nu_U(\{ -1\}) =0$.
It is apparent from (\ref{eq:defV2}) that the process behaves very differently if 
$\nu_U(\{ -1 \} ) > 0$. In the latter case the process 
restarts itself in finite exponential times from 0, therefore it cannot go to infinity 
regardless of the moment properties of the L\'evy measure. 
Indeed, as a consequence of a general result by Avrachenkov, Piunovskiy, and Zhang 
\cite{APZ} we show in Theorem \ref{thm:-1} below that in this case the process is always 
exponentially ergodic. Thus, concerning ergodic properties the case 
$\nu_U(\{ -1 \} ) = 0$ is more interesting, and we largely concentrate on it.
At the end of Section \ref{sect:main} we compare our results to earlier ones, and also 
spell out some statements in special cases.

Section \ref{sect:Fost-Lya} contains the description of the Foster--Lyapunov technique.
Using the explicit form of the generator of the process we give here the drift 
conditions corresponding to Theorems \ref{thm:genCD2}--\ref{thm:genCD3}. The 
infinitesimal generator of the process $V$ is determined in \cite{P, BL15}. The 
difficulty in our case is to show that domain of the extended generator contains unbounded 
norm-like functions; this is done in Proposition \ref{prop:extgen} in Section 
\ref{sect:proofs}. The proof of the drift conditions relies on Lemma \ref{lemma:integral}, 
which states that a two-dimensional integral with respect to the L\'evy measure $\nu_{UL}$ 
asymptotically equals to a one-dimensional integral with respect to the L\'evy measure 
$\nu_U$. This is the reason why the drift conditions in the theorems depend mainly on the 
univariate measure $\nu_U$. However, note that the integrability condition does depend on 
$\nu_{UL}$.
Finally, we investigate the petite sets. In Theorems \ref{thm:genCD2}--\ref{thm:genCD3} 
we assume that all compact sets are petite sets for some skeleton chain. In Proposition 
\ref{prop:petite-2} and in the remarks afterwards we give a sufficient condition for this 
assumption. It turns out that under natural conditions the petiteness assumption is 
satisfied.

The proofs are gathered in Section \ref{sect:proofs}. First we show that the domain of 
the extended generator is large enough. Then we deal with the drift conditions. After the 
short proof of the petiteness condition, the proofs of the main theorems are consequences 
of the drift conditions and general results in \cite{MT3, DFG}.

\section{Main results and discussion} \label{sect:main}

\subsection{Results}

We use the methods developed by Meyn and Tweedie \cite{MT1, MT2, MT3}, and we
also use their terminology. 
First, we recall some basic notions about Markov processes, which we need later.

As usual for a Markov process
$(X_t)_{t \geq 0}$ for any $x \in \R$, $\p_x$ and $\E_x$ stands for the probability and 
expectation conditioned on $X_0 = x$.
A time-homogeneous Markov process $(X_t)_{t \geq 0}$ on $\R$ is 
\textit{$\phi$-irreducible} (or simply irreducible), if for some $\sigma$-finite measure 
$\phi$ on $(\R, \mathcal{B}(\R))$, $\mathcal{B}(\R)$ being the Borel sets,
$\phi(B) > 0$ implies $\int_0^\infty \p_x \{ X_t \in B \} \d t > 0$, for all $x \in \R$.
The notion of \textit{petite sets} plays a crucial role in proving recurrence properties. 
A nonempty set $C \in \mathcal{B}(\R)$ is petite set (respect to the process $X$), if 
there is a probability distribution $a$ on $(0,\infty)$, and a nontrivial measure 
$\phi$, such that
\begin{equation*} \label{eq:def-petite}
\int_0^\infty \p_x\{ X_t \in A \} a(\d t) \geq \phi(A), \ \forall x \in C. 
\end{equation*}
The process $(X_t)_{t \geq 0}$ is called \textit{Feller process}, if 
$T_t f(x) := \E_x f(X_t) \in C_0$ for any $f \in C_0$, $t \geq 0$, 
and $\lim_{t \downarrow 0} T_t f(x) = f(x)$, for any $f \in C_0$, where 
$C_0 = \{ f: \, f \text{ is continuous}, \lim_{|x| \to \infty} f(x) = 0 \}$. If $T_t f$, 
$t > 0$, is only continuous, but does not necessarily tend to 0 at infinity, then the 
process is a \textit{weak Feller process}.

For a continuous time Markov process $(X_t)_{t \geq 0}$ the discretely sampled process 
$(X_{n\delta})_{n \in \N}$, $\delta > 0$, which is a Markov chain, called 
\textit{skeleton chain}. Irreducibility and petiteness are defined analogously for Markov 
chains. A Markov chain $(X_n)_{n \in \N}$ is \textit{weak Feller chain} if 
$T f(x) = \E_x f(X_1)$ is continuous and bounded for any continuous and bounded $f$.

For a measurable function $g \geq 1$ and a signed measure $\mu$ introduce the notation
\[
\| \mu \|_g = \sup \left\{ \int h \d \mu: |h| \leq g \right\}.
\]
When $g \equiv 1$ we obtain the total variation norm, which is simply denoted by
$\| \cdot \|$.

\smallskip

Before stating the main results we define the finite measure $\nu'$ on $[-1,1]$ by
\begin{equation} \label{eq:defnu'}
\nu'(A) = \nu_{UL}( (A \times \R) \cap \{ |z| > 1 \}), 
\end{equation}
where  $A \subset [-1,1]$ is Borel measurable.

In the theorems below we need that all compact sets are petite sets for 
some skeleton chain. This assumption is satisfied under mild conditions. Sufficient 
conditions are stated in Subsection \ref{subsect:petite}.

First we give a sufficient condition for the ergodicity of the process.

\begin{theorem} \label{thm:genCD2}
Assume that $\nu_U(\{ -1\} ) = 0$, 
all compact sets are petite for some skeleton chain,
\begin{equation} \label{eq:int-cond-CD2}
\dint_{|z| \geq 1} \log |z| \, \nu_{UL} (\d z) < \infty, \quad 
\int_{-3/2}^{-1/2} |\log |1 + z|  | \, \nu_U(\d z) < \infty,
\end{equation}
and
\begin{equation} \label{eq:ineqCD2}
\gamma_U - \frac{\sigma_U^2}{2}
+ \int_{\R} \left[  \log |1 + z|  - z I(|z| \leq 1) \right] \nu_U (\d z) + 
\int_{-1}^1 z \nu'(\d z) < 0.
\end{equation}
Then $V$ is ergodic, that is there is an invariant probability measure $\pi$, 
such that for any $x \in \R$
\[
\lim_{t \to \infty} \| \p_x \{ V_t \in \cdot \} - \pi\| =0. 
\]
\end{theorem}

In the L\'evy-driven OU case our assumptions reduce to the the necessary and sufficient 
condition for convergence to an invariant measure, given by Sato and Yamazato \cite{SY} 
(in any dimension); see Corollary \ref{cor:CD2} below.

\smallskip

Assuming stronger moment assumptions we obtain polynomial rate of convergence. However, 
note that the drift condition is the same as in the previous result.

\begin{theorem} \label{thm:gensubexp}
Assume that  $\nu_U(\{ -1\} ) = 0$, 
all compact sets are petite for some skeleton chain,
\begin{equation} \label{eq:int-cond-subexp}
\dint_{|z| \geq 1} (\log |z|)^\alpha \nu_{UL} (\d z) < \infty, 
\text{ for some } \alpha > 1, \quad  
\int_{-3/2}^{-1/2} |\log |1 + z| | \, \nu_U(\d z) < \infty,
\end{equation}
and (\ref{eq:ineqCD2}) holds. Then there is an invariant probability measure $\pi$, 
such that for some $C> 0$ for any $x \in \R$
\[
\| \p_x \{ V_t \in \cdot \} - \pi \| \leq  C \, (\log |x|)^\alpha \, t^{1-\alpha} .
\]
\end{theorem}

We can show not only polynomial but more general convergence rates. However, in the 
general case the assumptions are more complicated. We spell out one more example.
A general result on the drift condition is given in Proposition \ref{prop:CDsubexp-gen}, 
from its proof it will be clear, why the drift condition is the same in Theorems 
\ref{thm:genCD2}, \ref{thm:gensubexp}, and \ref{thm:gensubexp2}.

\begin{theorem} \label{thm:gensubexp2}
Assume that $\nu_U(\{ -1\} ) = 0$, 
all compact sets are petite for some skeleton chain,
\begin{equation} \label{eq:int-cond-subexp2}
\begin{gathered}
\dint_{|z| \geq 1} \exp \{ \gamma (\log |z|)^\alpha \} \nu_{UL} (\d z) < \infty, 
\text{ for some } \alpha \in (0,1), \gamma > 0, \\ 
\int_{-3/2}^{-1/2} | \log | 1 + z|  |\, \nu_U(\d z) < \infty,
\end{gathered}
\end{equation}
and (\ref{eq:ineqCD2}) holds.
Then there is an invariant probability measure $\pi$, 
such that for some $C> 0$ for any $x \in \R$
\[
\| \p_x \{ V_t \in \cdot \} - \pi \| \leq  
C \exp\{\gamma (\log |x|)^\alpha \} \,  e^{-(t/\alpha)^\alpha} t^{1-\alpha}.
\]
\end{theorem}

As in  Theorem 3.2 in \cite{DFG}, under the same assumptions as in Theorems 
\ref{thm:gensubexp} and \ref{thm:gensubexp2} above it is possible to prove convergence 
rates in other norms, 
i.e.~for $\| \p_x \{ V_t \in \cdot \} - \pi \|_g$ with specific $g$. There is a trade-off 
between the convergence rate and the norm function $g$: larger $g$ corresponds to weaker 
rate, and vice versa. See Theorem 3.2 \cite{DFG} and the remark after it.

Last we deal with exponential ergodicity.

\begin{theorem} \label{thm:genCD3}
Assume that $\nu_U(\{ -1\} ) = 0$, 
all compact sets are petite for some skeleton chain,
\begin{equation} \label{eq:int-condCD3}
\dint_{|z| \geq 1} |z|^\beta \nu_{UL} (\d z) < \infty,
\text{ for some } \beta \in (0,1],
\end{equation}
and
\begin{equation} \label{eq:ineqCD3}
\gamma_U - \frac{\sigma_U^2 (1-\beta)}{2}
+ \int_{\R}  \frac{|1 + z|^\beta - 1 - z \beta I(|z| \leq 1)}{\beta} \nu_U (\d z) + 
\int_{-1}^1 z \nu'(\d z) < 0.
\end{equation}
Then $V$ is exponentially ergodic, that is there is an 
invariant probability measure $\pi$, such that for some $\rho <1, C > 0$,
\[
\| \p_x \{ V_t \in \cdot \} - \pi\|_g \leq C (1 + |x|^\beta) \rho^t, 
\]
with $g(x) = 1 + |x|^\beta$.
\end{theorem}

Note that in the drift conditions (\ref{eq:ineqCD2}), (\ref{eq:ineqCD3})
only the last term depends on the joint measure $\nu_{UL}$, everything else is 
determined only by the law of $U$. However, the integral conditions do depend on the 
joint law of $(U,L)$.

\smallskip

Finally, consider the significantly different case $\nu_U(\{ -1\}) > 0$.
In this case the process returns to 0 in exponential times and the process restarts.
It is natural to expect that exponential ergodicity holds without further 
moment conditions. This is exactly the situation treated by Avrachenkov, Piunovskiy and 
Zhang \cite{APZ}. Put $\lambda = \nu_U (\{ -1\}) > 0$, and let $\widetilde V$ be 
the process with the same characteristics as $V$, except 
$\nu_{\widetilde U} (\{-1\}) = 0$.
Then the process $V$ can be seen as the process $\widetilde V$ which restarts from 0 
after independent exponential random times with parameter $\lambda$.
In Corollary 2.1 \cite{APZ} it is shown that 
\[
\pi (A ) = \int_0^\infty \p_0 \{ \widetilde V_t \in A \} \lambda e^{-\lambda s } \d s
\]
is the unique invariant probability measure for $V$. 
We summarize Theorem 2.2 in \cite{APZ} as follows.

\begin{theorem} \label{thm:-1}
Assume that $\lambda = \nu_U(\{-1\}) > 0$. Then the process $V$ is exponentially ergodic, 
that is
\[
\| \p_x \{ V_t \in \cdot \} - \pi \| \leq 2 e^{-\lambda t},
\]
where the invariant measure $\pi$ is defined above.
\end{theorem}

Note that in this case a strictly stationary causal solutions always exists, see 
\cite[Theorem 2.2]{BL15}.

\smallskip

Fasen \cite{F} investigated ergodic and mixing properties of GOU processes in the special 
case when $\eta$ is subordinator, and the stationary solution has a Pareto-like tail. 
Assuming that $V$ is simultaneously $\phi$-irreducible, and under some further moment 
assumptions Proposition 3.4 \cite {F} states that $V$ is exponentially  $\beta$-mixing, 
and exponentially ergodic. Lee \cite{L} proved exponential ergodicity and $\beta$-mixing 
for more general GOU processes. In Theorem 2.1 \cite{L} it is shown that the 
distribution of $(V_{nh})_{n \in \N}$ converges to a probability measure $\pi$, which is 
the unique invariant distribution for the process, if $0 < \E \xi_h \leq \E |\xi_h| < 
\infty$ and $\E \log^+ |\eta_h| < \infty$. Recall the definitions of $\eta$ and $\xi$ 
from (\ref{eq:def-eta}), (\ref{eq:def-xi}). The condition $\E |\xi_h| < \infty$ is much 
stronger than our condition in Theorem \ref{thm:genCD2}. However, when $U_t$ is 
continuous, one sees easily that $\E \xi_h = - h(\gamma_U - \sigma_U^2/2)$, and so 
conditions $\E \xi_h > 0$, $\E \log^+ |\eta_h|  < \infty$ are the same our
conditions in Theorem \ref{thm:genCD2}. Moreover, Lee showed in her Theorem 2.2 
that the exponential ergodicity holds for the $h$-skeleton process whenever for some $r > 
0$
\[
\E e^{-r \xi_h} < \infty, \text{ and }
\E \left| e^{-\xi_h} \int_0^h e^{\xi_{s-}} \d \eta_s \right|^{r} < \infty,
\]
and her Theorem 2.6 states that this implies exponential ergodicity for $V$ if the 
transition density functions exist, and they are uniformly bounded on compact sets. By 
Proposition 3.1 in \cite{B11} for $r \geq 1$ condition $\E e^{-r \xi_h} < \infty$ holds 
if and only if $\E |U_1|^r < \infty$. For the other condition note that by Proposition 
2.3 in \cite{LM}
\[
e^{-\xi_t} \int_0^t e^{\xi_{s-}} \d \eta_s \stackrel{\mathcal{D}}{=}
\int_0^t e^{-\xi_{s-}} \d L_s.
\]
Combining \cite[Proposition 4.1]{LM} (or rather its proof) and \cite[Proposition 
3.1]{B11} we have that the latter has finite $r$th moment, $r \geq 1$, if
$\E |U_1|^{\max \{1, r\} p } < \infty$, $\E  \mathcal{E}(U)_1^r < 1$, and 
$\E |\eta_1|^{\max \{1, r\} q } < \infty$, for some $p, q > 1$, $p^{-1} + q^{-1} =1$.

\subsection{Special cases}

In order to compare with existing results we spell out Theorem \ref{thm:genCD2} and 
\ref{thm:genCD3} in the L\'evy-driven case, when $U_t = - \mu t$. 

For ergodicity, we have the following

\begin{corollary} \label{cor:CD2}
Assume that all compact sets are petite sets for some skeleton chain, and the integral 
condition $\int_{|z| \geq 1} \log |z| \nu_L(\d z) < \infty$ holds. Then
$V$ is ergodic, that is there is an invariant probability measure $\pi$, 
such that for any $x \in \R$
\[
\lim_{t \to \infty} \| \p_x \{ V_t \in \cdot \} - \pi\| =0. 
\]
\end{corollary}

In the L\'evy-driven OU case the necessary and sufficient condition for convergence 
to an invariant measure was given by Sato and Yamazato \cite{SY} (in any 
dimension). They showed that $V_t$ converges in distribution if and only if 
$\int_{|z| \geq 1} \log |z| \nu_L(\d z) < \infty$, which is exactly our assumption.
Otherwise $|V_t|$ tends to infinity in probability. This suggests that the conditions in 
Theorem \ref{thm:genCD2} are optimal.

The exponential ergodicity reads as

\begin{corollary} \label{cor:LOU-CD3}
Assume that $\nu_U(\{ -1\} ) = 0$, all compact sets are petite sets for some skeleton 
chain, and $\int_{\R \backslash [-1,1]} |x|^{\beta} \nu_L(\d x) < \infty$
for some $\beta \in (0,1]$. Then $V$ is exponentially ergodic, 
that is there is an invariant probability measure $\pi$, such that for some 
$\rho<1, C > 0$
\begin{equation} \label{eq:LOU-exp}
\| \p_x \{ V_t \in \cdot \} - \pi\|_g \leq C (1 + |x|^\beta) \rho^t, 
\end{equation}
with $g(x) = 1 + |x|^\beta$.
\end{corollary}

For $d$-dimensional L\'evy-driven OU processes Masuda \cite{M} and Wang \cite{W} 
proved exponential ergodicity. We spell out their results in one dimension. Let
$U_t = - \mu t$, $\mu > 0$. 
Using Foster--Lyapunov techniques, in Theorem 4.3 \cite{M} Masuda proved that if
$\int_\R |x|^\beta \pi(\d x) < \infty$, $\beta > 0$, where $\pi$ is the stationary 
distribution, then $V_t$ is exponential $\beta$-mixing, i.e.~for some $\rho < 1, C>0$
\[
\int_\R \| \p_x \{ V_t \in \cdot \} - \pi \| \pi(\d x) \leq C \rho^t.
\]
Using coupling methods, Wang \cite[Theorem 1]{W} showed that (\ref{eq:LOU-exp}) holds 
with $\beta = 1$, i.e.~the process is exponentially ergodic, if
$\int_{\R \backslash [-1,1]} |x| \nu_L(\d x) < \infty$, and 
the L\'evy measure satisfies the smoothness condition
\[
\limsup_{r \downarrow 0} 
r^{-1} \sup_{|y| \leq r} \| \nu_L^\varepsilon - (\delta_y * \nu_L^\varepsilon) \| < 
\infty,
\]
where $\nu_L^\varepsilon(B) = \nu_L(B \backslash \{z : |z| < \varepsilon \} )$ if
$\nu_L(\R) = \infty$, otherwise $\nu_L^\varepsilon = \nu_L$,
$*$ stands for convolution, and $\delta_y$ is the Dirac-measure at $y$.
In Theorem 2  \cite{W} it was proved that (\ref{eq:LOU-exp}) holds if
$\int_{\R \backslash [-1,1]} |x|^{\beta} \nu_L(\d x) < \infty$, $\beta \in (0,1]$, and 
the L\'evy measure satisfies
\[
\liminf_{|x| \to \infty} 
\frac{x^2 \int_{|z| \leq 1/|x|} z^2 \nu_L(\d z) }{\log (1 + |x|)} > 0.
\]
The latter condition implies $\nu_L(\R) = \infty$, and it is satisfied for stable 
processes. Thus our exponential ergodicity 
results are new even in the L\'evy driven case.

Subexponential rates are rare in the literature. For compound Poisson driven
Ornstein--Uhlenbeck processes with nonnegative step size Fort and Roberts 
\cite[Lemma 18]{FR} proved polynomial rate of convergence, while in the same setup 
(under stronger moment conditions) Douc, Fort, and Guillin \cite[Proposition 5.7]{DFG} 
showed more general subexponential convergence rates. Theorems \ref{thm:gensubexp} and 
\ref{thm:gensubexp2} are generalizations of their results.

\smallskip

Our results are optimal in the following sense.
Fort and Roberts gave examples for a compound Poisson-driven OU-process, which fails to 
be exponentially ergodic, or even is not positive recurrent \cite[Example 3.3]{FR}.
Assume that $V$ is a L\'evy-driven OU process such that $L_t = S_{N_t}$, 
where $N_t$ is a standard Poisson process, and $S_n = X_1 + \ldots + X_n$, where
$X, X_1, \ldots$ are i.i.d.~nonnegative random variables. In \cite[Lemma 17]{FR} it was 
shown that if $\E X^\alpha = \infty$ for any $\alpha > 0$, then the process is not 
exponentially ergodic. Moreover, if $\E \log X = \infty$ then the process is not positive 
recurrent.

\medskip

Finally, we spell out some of the results in the continuous case, i.e.~when
the process $V$ is a diffusion, which is much easier to handle. More importantly, in this 
case we are able to show some negative results, e.g.~when $V$ is nonrecurrent.
We use the results in Khasminskii \cite{K}.

Let $(U, L)$ be a bivariate Brownian motion with generating triplet
$(\gamma_Z, \Sigma_Z, 0)$. Let $A=(a_{i,j})_{i,j=1,2}$ be a matrix, such that
$A A^\top = \Sigma_Z$, and let $W, \widetilde W$ be independent standard
Brownian motions. Then $(U, L)$ can be written as
\[
\begin{split}
U_t = & \gamma_U t + a_{11} W_t + a_{12} \widetilde W_t, \\
L_t = & \gamma_L t + a_{21} W_t + a_{22} \widetilde W_t.
\end{split}
\]
The SDE defining $V$ reads as
\begin{equation*} \label{eq:SDE-cont}
\begin{split}
\d V_t & = V_t \d U_t + \d L_t \\ 
& = (\gamma_U V_t + \gamma_L) \d t + (a_{11} V_t + a_{21}) \d W_t
+ ( a_{12} V_t + a_{22} ) \d \widetilde W_t.
\end{split}
\end{equation*}
From standard results we see that there exists a unique solution on $(0,\infty)$.
By Theorem 3.1 in \cite{BL15} the infinitesimal
generator of the process $V$ is
\[
\mathcal{A} f(x) = (\gamma_U x + \gamma_L) f'(x) +
\frac{1}{2} ( x^2 \sigma_U^2 + 2 x \sigma_{UL} + \sigma_L^2) f''(x). 
\]
The main difference compared to the general case is that this operator is a local 
operator, therefore the domain of the extended generator automatically contains all $C^2$ 
functions. (See (\ref{eq:def-A}) and Proposition \ref{prop:extgen}.)

We use the terminology of Khasminskii \cite{K}.
A process is \emph{recurrent relative to the domain $U$} if it is regular
(defined on $(0,\infty)$) and for every $ x \in U^c$ one has
$\p_x \{ \tau_U < \infty \} =1$, where $\tau_U$ is the first entrance time into $U$. It 
is \emph{positive recurrent} if $\E_x \tau_U < \infty$. If a homogeneous process is 
recurrent relative to some bounded open domain, then it is recurrent to any bounded open 
domain, see \cite[Lemma 4.1, p.~101]{K}. 

The first part of the next statement follows from our 
general results. On the other hand, in this special case we have the converse.

\begin{proposition} \label{prop:trans}
Whenever $\gamma_U < \sigma_U^2/2$ the process $V$ is positive recurrent. While if
$\gamma_U > \sigma_U^2/2$ then $V$ is nonrecurrent relative to any bounded open domain 
$U$.
\end{proposition}

%
%
%
%
%

\section{The Foster--Lyapunov method} \label{sect:Fost-Lya}

In this section we describe the Foster--Lyapunov method.

\subsection{Drift conditions}

We again start with some notation.
The \textit{infinitesimal generator} $\mathcal{A}$ of a the Markov process $X$ is defined 
as 
\[
\mathcal{A}f(x) = \lim_{t \downarrow 0} t^{-1} \E_x [ f(X_t) - f(x) ]
\]
whenever it 
exists. Its domain is denoted by $\mathcal{DI}(X)$.
The \textit{extended generator} $\mathcal{A}$ of the Markov process $X$ is defined as 
$\mathcal{A}f =g$ whenever $f(X_t) - \int_0^t g(X_s) \d s$ is a local martingale with 
respect to the natural filtration. Its domain is denoted by $\mathcal{DE}(X)$. The same 
notation should not cause confusion, since the two operators are the same, only the 
domains are different.

In order to apply Foster--Lyapunov techniques we have to truncate the process 
$V$. For $n \in \N$ let
\begin{equation} \label{eq:V-trunc}
V_t^n = V_{t \wedge T^n}
\end{equation}
where $T^n = \inf \{t \geq 0 : |V_t| \geq n  \}$. Note that this is not exactly the 
process defined in \cite[p.521]{MT3}, but the results in \cite{MT3} are valid for our 
process; see the comment after formula (2) in \cite[p.521]{MT3}. We also emphasize that 
the stopped process is not necessarily bounded.

Let us define the operator $\mathcal{A}$ as 
\begin{equation} \label{eq:def-A}
\begin{split}
\mathcal{A} f(x) & =  
(x \gamma_U + \gamma_L) f'(x) + \frac{1}{2} ( x^2 \sigma_U^2 + 
2 x \sigma_{UL} + \sigma_L^2) f''(x) \\
& \phantom{=} \, + \dint_{\R^2} \left[
f(x + x z_1 + z_2) - f(x) - f'(x) (x z_1 + z_2) I(|z| \leq 1) \right] \nu_{UL}(\d z),
\end{split}
\end{equation}
where $f \in C^2$ (the set of twice continuously differentiable functions)
is such that the integral in the definition exists.
In Exercise V.7 in Protter  \cite{P} and in Theorem 3.1 in Behme and Lindner 
\cite{BL15} it is shown that the infinitesimal generator of $V$ is $\mathcal{A}$, and
$C^\infty_c \subset \mathcal{DI}(V)$, where
$C^\infty_c$ is the set of infinitely many times differentiable compactly supported 
functions. Moreover, if $\nu_U(\{ -1 \} ) = 0$ then $V$ is a Feller process, and
$
\mathcal{DI}(V) \supset 
\{ f \in C_0^2 : \lim_{|x| \to \infty} ( |x f'(x)| + x^2 |f''(x)|) = 0 \},
$
which is a core, see \cite[Theorem 3.1]{BL15}. Here
$C_0^2= \{ f: \, f \text{ twice continuously differentiable, and }
\lim_{|x| \to \infty} f''(x) = 0 \}$. It is clear from the regenerative property of the 
process in (\ref{eq:defV2}) that if $\nu_U(\{ -1\} ) > 0$ then it is \emph{not} 
Feller process, only weak Feller.
A slightly different form of the generator in terms of $(\xi, \eta)$, for independent 
$\xi$ and $\eta$ is given in \cite[Proposition 2.3]{KPS}, see also \cite[Remark 
3.4]{BL15}.

Let us define the generator of $V^n$ as
\[
\mathcal{A}_n f(x) = 
\begin{cases}
\mathcal{A} f(x), & |x| < n, \\
0, & |x| \geq n.
\end{cases}
\]
In Proposition \ref{prop:extgen} we show that $\mathcal{A}_n$ is indeed the extended 
generator of the process $V^n$, and
\[
\mathcal{DE}(V^n) \supset 
\left\{ f \in C^2 : \, \dint_{|z| > 1} |f(|z|)| \nu_{UL}(\d z) < \infty \right\}.
\]
From this result it also follows that 
\[
\mathcal{DE}(V) \supset 
\left\{ f \in C^2 : \, \dint_{|z| > 1} |f(|z|)| \nu_{UL}(\d z) < \infty \right\}.
\]

Following \cite{MT3}
we introduce the various ergodicity conditions for the generator $\mathcal{A}_n$.
A function $f: \R \to [0,\infty)$ is \textit{norm-like} if $f(x) \to \infty$ as
$|x| \to \infty$. In the following conditions below $f$ is always a norm-like 
function.
The  recurrence condition is
\begin{equation} \label{eq:CD1}
\exists f, \, \exists d > 0, \, C \text{\ compact, such that }
\mathcal{A}_n f(x) \leq d I_C(x), \ |x| < n,  n \in \N.
\end{equation}
The ergodicity condition is
\begin{equation} \label{eq:CD2}
\begin{gathered}
\exists f, \, \exists c, d > 0, \, g \geq 1 \text{ measurable}, \,
C \text{ compact, such that}\\
\mathcal{A}_n f(x) \leq -c g(x) + d I_C(x), \ |x| < n,  n \in \N.
\end{gathered}
\end{equation}
The exponential ergodicity condition is
\begin{equation} \label{eq:CD3}
\exists f, \, \exists c, d > 0  \text{, such that } 
\mathcal{A}_n f(x) \leq -c f(x) + d,  \ |x| < n,  n \in \N.
\end{equation}
For subexponential rates of convergence we use 
more recent results due to Douc, Fort and Guillin \cite{DFG},
Bakry, Cattiaux and Guillin \cite{BCG}. For a survey see also
Hairer's notes \cite{H}. The 
subexponential ergodicity condition 
(\cite[Theorems 3.4 and 3.2]{DFG}) is
\begin{equation} \label{eq:CDsubexp}
\exists f \geq 1, d > 0, \, C \text{ compact},
\varphi  \text{ positive concave, such that } 
\mathcal{A} f(x) \leq - \varphi(f(x)) + d I_C(x). 
\end{equation}

In the following we state the drift conditions corresponding to
Theorems \ref{thm:genCD2}, \ref{thm:gensubexp},  \ref{thm:gensubexp2}, and 
\ref{thm:genCD3}, respectively.

\begin{proposition} \label{prop:genCD2}
Assume the integrability condition (\ref{eq:int-cond-CD2}),
and the drift condition (\ref{eq:ineqCD2}).
Then (\ref{eq:CD2}) holds with $f(x) = \log |x|$, $|x| \geq e$, and $g\equiv 1$.
\end{proposition}

\begin{proposition} \label{prop:genCDsubexp}
Assume the integrability condition (\ref{eq:int-cond-subexp})
and the drift condition (\ref{eq:ineqCD2}).
Then (\ref{eq:CDsubexp}) holds with $f(x) = (\log |x|)^\alpha$, $|x|\geq 3$, and
$\varphi(x)= x^{1-1/\alpha}$.
\end{proposition}

\begin{proposition} \label{prop:genCDsubexp2}
Assume that (\ref{eq:int-cond-subexp2}) and (\ref{eq:ineqCD2}) hold.
Then (\ref{eq:CDsubexp}) holds with $f(x) = \exp \{ \gamma (\log |x|)^\alpha\}$, 
$|x| \geq e$, and $\varphi(x)= x \, (\log x)^{1-1/\alpha}$.
\end{proposition}

\begin{proposition} \label{prop:genCD3}
Assume that (\ref{eq:int-condCD3}) and (\ref{eq:ineqCD3}) hold.
Then (\ref{eq:CD3}) holds with $f(|x|) = |x|^\beta$, $|x| \geq 1$.
\end{proposition}

\medskip

Without any petiteness condition (\ref{eq:CD1}) implies that the process is 
nonevanescent (Theorem 3.1 \cite{MT3}). Moreover, according to Theorem 4.5 \cite{MT3} 
condition (\ref{eq:CD2}) together with the weak Feller property implies the 
existence of an invariant probability measure $\pi$.

\subsection{Petite sets} \label{subsect:petite}

In this subsection we give sufficient condition for all compact sets to be petite sets 
for some skeleton chain. Under natural assumptions this condition holds.
%
%

By investigating ergodicity rates a minimal necessary assumption is that the process 
converges in distribution. If $V_t$ converges in distribution for 
any initial value $x_0 \in \R$ then $V$ is $\pi$-irreducible, where $\pi$ is the law of 
the limit distribution. Certain properties of the limit distribution imply that compact 
sets are petite sets. Recall the relation $(U,L)$ and $(\eta, \xi)$ from 
(\ref{eq:def-eta}), (\ref{eq:def-xi}).

%

\begin{proposition} \label{prop:petite-2}
Assume that $\lim_{t \to \infty} \xi_t = \infty$ a.s., 
$\int_0^\infty e^{-\xi_{s-}} \d L_s $ exists a.s., and its distribution $\pi$ is 
such that the interior of its support is not empty. Then all compact sets are 
petite sets for the skeleton chain $(V_n)_{n \in \N}$.
\end{proposition}

The assumptions above imply the existence of a stationary causal solution to 
(\ref{eq:SDEGOU}). Moreover, for any initial value $x_0$ the solution converges in 
distribution to the stationary solution; see \cite[Theorem 2.1 (a)]{BLM}. Necessary and 
sufficient conditions for $\lim_{t \to \infty} \xi_t = \infty$ a.s., and for 
the existence of $\int_0^\infty e^{\xi_{s-}} \d L_s $ are given in \cite[Theorem 
3.5 and 3.6]{BLM}.

In the L\'evy-driven OU case, when $U_t = - \mu t$, $\mu >0$, the convergence holds 
if and only if $\int_{|z| > 1} \log |z| \nu_L(\d z) < \infty$, in which case the limit 
distribution is self-decomposable; see \cite[Theorem 4.1]{SY}. A nondegenerate 
selfdecomposable distribution has absolute continuous density with respect to the 
Lebesgue measure, in particular the interior of its support is not empty;
see \cite[Theorem 28.4]{S}.

In the general case, there is much less known about the properties of the integral 
$J= \int_0^\infty e^{-\xi_{s-}} \d L_s$. Continuity properties of these integrals were 
investigated in \cite{BLM08}. It was shown in \cite[Theorem 2.2]{BLM08} that if $\pi$ has 
an atom then it is necessarily degenerate. In particular, if $\xi$ is 
spectrally negative (does not have positive jumps), then $\pi$ is still self-decomposable; 
see \cite{BLM08} Theorem 2.2, and the remark after it. 
When $L_t = t$ sufficient conditions for the existence of the density of
$\int_0^\infty e^{-\xi_{s}} \d s$ were given in \cite[Proposition 2.1]{CPY2}.
When $L$ and $U$ are independent, $\E |\xi_1| < \infty$, $\E \xi_1 > 0$, 
$\E |\eta_1 | < \infty$, 
and $\sigma_U^2 + \sigma_L^2 > 0$ then $J$  has 
continuously differentiable density; see \cite[Corollary 2.5]{KPS}.
The case, when $(\xi_t, L_t)_{t \geq 0} = ((\log c) N_t, Y_t)$, $c >1$, where $(N_t)_{t 
\geq 0}, (Y_t)_{t \geq 0}$ are Poisson processes, and $(N_t, Y_t)_{t \geq 0}$ is a 
bivariate L\'evy process was treated in \cite{LS}. 
Whether the distribution of $J$ is absolute continuous or continuous singular depends 
on algebraic properties of the constant $c$, see Theorems 3.1 and 3.2 \cite{LS}.
The problem of absolute continuity in this case is closely related to infinite Bernoulli 
convolutions; see Peres, Schlag, and Solomyak \cite{PSS}.
For further results in this direction we refer to
\cite{BLM08, LS, KPS} and the references therein.

%
%

\section{Proofs} \label{sect:proofs}

First we show that the domain of the extended generator is large enough, and contains 
usual norm-like functions, which are not bounded. In Subsection \ref{subsect:proof-drift} 
after some preliminary technical lemmas we prove that the various drift conditions hold. 
Subsection \ref{subsect:proof-petite} contains the proof of the sufficient condition for 
the petiteness assumption. Finally, we prove the main theorems, which are easy 
consequences of the drift conditions and some general results from \cite{MT1, MT3, DFG}.

\subsection{Extended generator and infinitesimal generator}

In the following $(\mathcal{F}_t)_{t \geq 0}$ stands for the natural filtration induced 
by the bivariate L\'evy process $(U,L)$. Martingales are meant to be martingales with 
respect to $(\mathcal{F}_t)_{t \geq 0}$.

\begin{proposition} \label{prop:extgen}
Assume that $f \in C^2$, and for each fixed $n \in \N$,
\begin{equation} \label{eq:extgen-cond}
\sup_{|x| \leq n} \dint_{|x+x z_1 + z_2| > m}
\left( 1 + | f(x+x z_1 + z_2)| \right) \nu_{UL}(\d z) =: \eta_m^n < \infty,
\end{equation}
and $\lim_{m \to \infty} \eta_m^n = 0$.
Then $f \in \mathcal{DE}(V^n)$, $n \in \N$, and $f \in \mathcal{DE}(V)$.
\end{proposition}

We use this proposition for norm-like functions $f$, which for $|x|$ large enough 
equals to $(\log |x|)^\alpha$, $\alpha \geq 1$, $\exp\{ \gamma (\log x)^\alpha\}$, 
$\gamma > 0, \alpha \in (0,1)$, or $|x|^\beta$, $\beta \in (0,1]$. For these `nice' 
functions 
(\ref{eq:extgen-cond}) is satisfied when
$\dint_{\R^2} f(|z|) \nu_{UL}(\d z) < \infty$.

\begin{proof}
First let $f \in \mathcal{DI}(V)$. It is well-known that
\[
M_t = f(V_t) - \int_0^t \mathcal{A}f(V_s) \d s, \quad t \geq 0, 
\]
is martingale. Consider the stopping time $T_n = \inf\{ t \geq 0: |V_t | \geq n\}$, then
by (\ref{eq:V-trunc})
\begin{equation} \label{eq:stop}
\begin{split}
M_{t \wedge T_n} 
&  =  f(V_{t \wedge T_n}) - \int_0^{t \wedge T_n} \mathcal{A}f(V_s) \d s \\
&  =  f(V_t^n) - \int_0^{t \wedge T_n} \mathcal{A}_n f(V_s) \d s \\
&  =  f(V_t^n) - \int_0^{t} \mathcal{A}_n f(V_s^n) \d s,
\end{split}
\end{equation}
where we used that $|V_{s-}| < n$ if and only if $s \leq T_n$. Since $M_{t \wedge T_n}$ is 
a 
martingale, we have proved that $\mathcal{DI}(V) \subset \mathcal{DE}(V^n)$.
\smallskip 

Now we handle the general case. We may and do assume that $f$ is nonnegative.
Consider a sequence of nonnegative functions 
$\{ g_m \} \subset \mathcal{DI}(V)$ with the following properties: $g_m(x) \equiv f(x)$ 
for $|x| \leq m$, and $\equiv 0$ for $|x| \geq m+1$,
$\max_{x \in \R} g_m(x) \leq \sup_{x \in [-m,m]} f(x) +1$, and
$g_m \leq g_{m+1}$. Define the martingales
\[
M_m(t) = g_m(V_t^n) - \int_0^t \mathcal{A}_n g_m(V_{s}^n) \d s. 
\]
Let $ \ell \geq m$ and to ease the notation put $h(x) = g_\ell(x) - g_m(x)$. Since
$h(x) \equiv 0$ for $|x| \leq m$ and for $|x| \geq \ell + 1$ we have for 
$|x| \leq n < m$
\[
\begin{split}
\mathcal{A}_n h(x) & =  
(x \gamma_U + \gamma_L) h'(x) + \frac{1}{2} ( x^2 \sigma_U^2 + 
2 x \sigma_{UL} + \sigma_L^2) h''(x) \\
& \phantom{=} \, + \dint_{\R^2} \left[
h(x + x z_1 + z_2) - h(x) - h'(x) (x z_1 + z_2) I(|z| \leq 1) \right] \nu_{UL}(\d z) \\
& = \dint_{\R^2} h(x + x z_1 + z_2) \nu_{UL}(\d z).
\end{split} 
\]
Thus for all $|x| < n$
\[
| \mathcal{A}_n h(x) | \leq  \eta_m = \eta_m^n,
\]
therefore we have
\begin{equation} \label{eq:Aineq1}
\left| \int_0^t \mathcal{A}_n [g_\ell(V_{s}^n) - g_m(V_{s}^n)] \d s \right| \leq 
t \eta_m.
\end{equation}
Using  $g_\ell \geq g_m$ and the martingale property
\[
\begin{split}
\E | g_\ell(V_t^n) - g_m(V_t^n) | & = 
\E [   g_\ell(V_t^n) - g_m(V_t^n) ] \\
& = \E \left( M_\ell(t) - M_m(t) + 
\int_0^t \mathcal{A}_n [g_\ell(V_{s}^n) - g_m(V_{s}^n)] \d s \right) \\
& = \E \int_0^t \mathcal{A}_n [g_\ell(V_{s}^n) - g_m(V_{s}^n)] \d s,
\end{split}
\]
thus
\[
\E | g_\ell(V_t^n) - g_m(V_t^n) | \leq t \eta_m.
\]
Letting $\ell \to \infty$ Fatou's lemma gives
\[
\E [ f(V_t^n) - g_m(V_t^n) ] \leq t \eta_m. 
\]
Moreover, as in (\ref{eq:Aineq1})
\[
\left| \int_0^t \mathcal{A}_n [ f(V_{s}^n) - g_m(V_{s}^n) ] \d s \right| \leq t \eta_m.
\]
Thus we obtain for each $t \geq 0$
\[
M_m(t) \to f(V_t^n) - \int_0^t \mathcal{A}_n f(V_s^n) \d s =: M(t), 
\quad \textrm{a.s.~and in } L^1.
\]
Since $M_m$ is martingale, we have for each $0 \leq u < t$
\[
\begin{split}
& \E \left| M(u) - \E [ M(t) | \mathcal{F}_u ] \right| \\
& \leq 
\E | M(u) - M_m(u) | + \E \left| M_m(u) -  \E [ M_m(t) | \mathcal{F}_u ] \right| 
+ \E \left| \E [ M_m(t) - M(t) | \mathcal{F}_u ] \right| \\
&\leq  2 u \eta_m + \E | M_m(t) - M(t)|
\leq 4 t \eta_m \to 0,
\end{split}
\]
as $m \to \infty$. Thus 
$\E [ M(t) | \mathcal{F}_u ] = M(u)$ a.s., i.e.~$M$ is a martingale, and
$f \in \mathcal{DE}(V^n)$.

Finally, (\ref{eq:stop}) shows that $f(V_t) - \int_0^t \mathcal{A}f(V_s) \d s$ is a local 
martingale with localizing sequence $T_n$.
\end{proof}

\subsection{Drift conditions} \label{subsect:proof-drift}

We frequently use the following technical lemma. Recall the definition $\nu'$ from
(\ref{eq:defnu'}).

\begin{lemma} \label{lemma:integral}
Let $f$ be a norm-like $C^2$ function, such that $f$ is even, there exists $k_f>0$ 
such that $f(x+y) \leq k_f +f(x) + f(y)$, $x,y \in \R$, 
$\lim_{x \to \infty} x \sup_{|y| \geq x} |f''(y)| = 0$,
$\lim_{x \to \infty} f'(x) =0$, and 
$\dint_{\R^2} f(|z|) \nu_{UL}(\d z) < \infty$. Assume that $\nu_U(\{-1\}) =0$. Then
\[
\begin{split}
&\dint_{\R^2} [ f(x+x z_1 +z_2) - f(x) - f'(x) (x z_1 +  z_2) I(|z| \leq 1) ]
\nu_{UL}(\d z) \\
&=
\int_{\R} [ f(x+x z) - f(x) - f'(x) x z I(|z| \leq 1) ] \nu_U(\d z) + 
f'(x) x \int_{-1}^1 z \nu'(\d z) + o(1),
\end{split}
\]
where $o(1) \to 0$ as $|x| \to \infty$. Moreover, the same holds with $O(1)$ when $f'(x)$ 
is bounded.
\end{lemma}

\begin{proof}
We may write
\begin{equation*} \label{eq:int-decomp1}
\begin{split}
&\dint_{\R^2} [ f(x+x z_1 +z_2) - f(x) - f'(x) (x z_1 +  z_2) I(|z| \leq 1) ]
\nu_{UL}(\d z) \\
& = \dint_{\R^2} [ f(x+x z_1 +z_2) - f(x+ x z_1) - f'(x) z_2  I(|z| \leq 1)]
\nu_{UL}(\d z) \\
& \phantom{=} \,  + \dint_{\R^2} [ f(x+x z_1) - f(x) - f'(x) x z_1  I(|z| \leq 1) ] 
\nu_{UL}(\d z) \\
& =: I_1 + I_2,
\end{split}
\end{equation*}
since an application of the mean value theorem implies that both integral above is 
finite. Indeed, for the integrand in $I_1$ we have for $x$ large enough
\begin{equation*} \label{eq:int-est1}
\left| f(x+x z_1 +z_2) - f(x+ x z_1) - f'(x) z_2 \right| \leq 
|z_2| (|x z_1| + |z_2| ) \max_{y} |f''(y)|,
\end{equation*}
which implies the integrability. Let $I_1 = I_{11} + I_{12}$, where $I_{11}$ stands for 
the integral on $\{ |z| \leq 1\}$, and $I_{12}$ on $\{ |z| > 1\}$.
Let $\delta > 0$ be arbitrary. For $z_1 > -1 + \delta$
\[
\begin{split}
& \dint_{|z| \leq 1, z_1 > -1+\delta}  
[ f(x+x z_1 +z_2) - f(x+x z_1) - f'(x) z_2 I(|z| \leq 1) ] \nu_{UL}(\d z) \\
& \leq \sup_{|y| \geq \delta |x|-1} |f''(y)|
\dint_{|z| \leq 1, z_1 > -1+\delta} |z_2| (|x z_1| + |z_2| )  \nu_{UL}(\d z) ,
\end{split}
\]
which goes to 0, since $ x \sup_{|y| \geq x} f''(y) \to 0$.
Cutting further the remaining set
\[
\begin{split}
& \left| \dint_{|z| \leq 1, z_1 < -1+\delta}
[ f(x+x z_1 +z_2) - f(x+ x z_1) - f'(x) z_2  I(|z| \leq 1)] \nu_{UL}(\d z) \right| \\
& \leq |f'(x)| \dint_{|z| \leq 1, z_1 < -1+\delta} |z_2| \nu_{UL}(\d z) 
+ \max_{y \in \R } |f'(y)| \dint_{|z| \leq 1, z_1 < -1+ \delta} 
|z_2| \nu_{UL}(\d z).
\end{split}
\]
Since $\delta > 0$ can be arbitrarily small we have that
$I_{11} = o(1)$.

To handle $I_{12}$ fix $\varepsilon > 0$. There is an $R > 0$ such that
$\dint_{|z| > R} [k_f + f(z_2)] \nu_{UL}(\d z) < \varepsilon$, and so
\[
\left| \dint_{|z| > R} [ f(x+x z_1 +z_2) - f(x+ x z_1) ] \nu_{UL}(\d z) \right|
\leq \dint_{|z| > R}  [k_f + f(z_2)] \nu_{UL}(\d z) < \varepsilon.
\]
Let us choose $\delta = \delta(R)$ so small that
\[
\dint_{|z| \leq R, |z_1 + 1| \leq \delta} |z_2| \nu_{UL}(\d z) < \varepsilon.
\]
This is possible, since $\nu_U(\{ -1 \} ) =0$. Then we have
\[
\begin{split}
|I_{12}  | & \leq \max_{|y| > \delta |x| - R} |f'(y)| 
\dint_{|z| \leq R, |z_1 + 1 | > \delta} |z_2| \nu_{UL}(\d z) + 
\max_{y \in \R} |f'(y)| \dint_{|z| \leq R, |z_1 + 1 | \leq \delta} |z_2| \nu_{UL}(\d z) \\
& \phantom{=} \, +  \dint_{|z| > R} [k_f + f(z_2)] \nu_{UL}(\d z) \\
& \leq \max_{|y| > \delta |x| -R} |f'(y)| 
\dint_{|z| \leq R, |z_1 + 1 | > \delta} |z_2| \nu_{UL}(\d z) +
\varepsilon ( 1 + \max_{y \in \R} |f'(y)|), 
\end{split}
\]
which proves that $I_1 = o(1)$, if $f'(x) \to 0$.
We also see that if $f'$ is only bounded then $I_{1} = O(1)$.

We turn to $I_2$.
Note that in the integrand in $I_2$ only the indicator depends on $z_2$. Put
$\nu^1(A) = \nu_{UL} ( (A \times \R) \cap \{ |z| \leq 1 \} )$,
$\nu^{2}(A) = \nu_{UL} ( (A \times \R) \cap \{ |z| > 1 \})$.
Then $\nu^2$ is a finite measure on $\R$ and
$\nu^2|_{\{ |z| > 1\} } \equiv \nu_U|_{\{ |z| > 1 \} }$,  
$\nu^2 = \nu' + \nu_U|_{\{ |z| > 1\} }$, and $\nu_1 + \nu' = \nu_U|_{\{ |z| \leq 1\} }$.
Thus
\[
\begin{split}
I_2 & = 
\int_{-1}^1 [ f(x+x z) - f(x) - f'(x) x z  ] \nu^1(\d z)
+ \int_{\R} [ f(x+x z) - f(x) ] \nu^2(\d z) \\
& =\int_{\R} [ f(x+x z) - f(x) - f'(x) x z  I(|z| \leq 1) ] \nu_{U}(\d z)
+ f'(x) x \int_{-1}^1 z \nu'(\d z),
\end{split}
\]
and the statement is proven.
\end{proof}

The following simple statement shows that if $f(x)$ is concave for large $x$, then
$f(x+y) \leq k_f + f(x) + f(y)$, for some $k_f > 0$. That is, whenever the integrability 
condition holds, our norm-like functions ($(\log |x|)^\alpha$, $\alpha \geq 1$, 
$|x|^\beta$, $\beta \in (0,1]$) satisfy the condition of Lemma \ref{lemma:integral}
and Proposition \ref{prop:extgen}. 
The lemma follows from simple properties of concave functions. We omit the proof.

\begin{lemma} \label{lemma:concave}
Assume that $f: [0,\infty) \to [0,\infty)$ is concave on the interval $[x_0,\infty)$. 
Then there is a $k_f > 0$ such that $f(x + y) \leq k_f + f(x) + f(y)$ for all
$x, y \geq 0$.
\end{lemma}

\begin{proof}[Proof of Proposition \ref{prop:genCD2}]
Let $f(x) = \log |x|$ for $|x| \geq e$, and consider a smooth 
nonnegative extension of it to $[-e,e]$. Since $\log x$ is concave
$f$ satisfies the assumptions of Lemma \ref{lemma:integral}.

We have
\[
\int_{|z+1| > e/|x|} [ f(x + x z) - f(x) - f'(x) x z I(|z| \leq 1)] \nu_U (\d z)
= \int_{|z+1| > e/|x|} [ \log |1 + z|  - z I (|z| \leq 1) ]\nu_U (\d z),
\]
and
\[
\left| \int_{-1-e/|x|}^{-1+e/|x|} [ f(x + x z) - f(x) - f'(x) x z I(|z| \leq 1)]
\nu_U (\d z) 
\right|
\leq (\log x + 2) \nu_U ([-1-e/|x|, -1+e/|x|]).
\]
The integrability condition 
$\int_{-3/2}^{-1/2} | \log | 1 + z | | \, \nu_U(\d z) < \infty$ implies that the latter 
bound tends to 0 as $|x| \to \infty$. Therefore
\[
\int_{\R} [ f(x + x z) - f(x) - f'(x) x z I(|z| \leq 1)] \nu_U (\d z)
= \int_{\R} [ \log |1 + z|  - z  I(|z| \leq 1)] \nu_U (\d z)
+ o(1).
\]

Summarizing, we obtain
\begin{equation*} \label{eq:int-decomp3}
\begin{split}
&\dint_{\R^2} [ f(x+x z_1 +z_2) - f(x) - f'(x) (x z_1 +  z_2) I(|z| \leq 1) ]
\nu_{UL}(\d z) \\
& = 
\int_{\R}  [ \log |1 + z| - z I(|z| \leq 1) ] \nu_U (\d z) 
+ \int_{-1}^1 z \nu'(\d z) + o(1).
\end{split}
\end{equation*}
Therefore
\[
\mathcal{A}_n f(x) =  
\gamma_U - \frac{\sigma_U^2}{2} + 
\int_{\R}  [ \log |1 + z| - z I(|z| \leq 1) ] \nu_U (\d z) 
+ \int_{-1}^1 z \nu'(\d z) + o(1),
\]
so (\ref{eq:CD2}) holds.
\end{proof}

Before the proofs of Propositions \ref{prop:genCDsubexp} and \ref{prop:genCDsubexp2} 
we show a more general statement, which explains why the drift condition in 
Theorems \ref{thm:genCD2}, \ref{thm:gensubexp}, and \ref{thm:gensubexp2} is the same.

\begin{proposition} \label{prop:CDsubexp-gen}
Assume that $f$ satisfies the conditions in Lemma \ref{lemma:integral},
$g(x) = x f'(x)$ is positive, slowly varying at infinity,
increases for $x > x_0 > 0$, there is a $k_g > 0$, such that $g(xu) \leq k_g g(x) g(u)$ 
for all $u, x \geq 1$, $x f''(x) \sim - f'(x)$, and there is a concave function $\varphi$ 
such that $\varphi(f(x)) = g(x)$ for $x$ large enough.  Furthermore, assume that 
$\int_{|z| > 1}  g(z) \, \log |z| \, \nu_U(\d z) < \infty$,
$\int_{-3/2}^{-1/2} |1+z|^{-\varepsilon} \, \nu_U(\d z) < \infty$ for 
some $\varepsilon > 0$, and (\ref{eq:ineqCD2}) holds.
Then (\ref{eq:CDsubexp}) holds. 
\end{proposition}

We note that whenever $g(x) = x f'(x)$ is slowly varying, $f(x)$ is also slowly 
varying, and $f(x) / g(x) \to \infty$; see \cite[Proposition 1.5.9a]{BGT}.
This is the reason why we have to assume some integrability assumptions around $-1$. In 
particular, when $\nu_U(\{ -1 \} ) > 0$ the situation is completely different.
Moreover, the same argument shows that if $x g'(x)$ is slowly varying, then $g(x)$ is 
slowly varying, and $g(x)/(x g'(x)) \to \infty$, which implies that 
$x f''(x) \sim - f'(x)$.

\begin{proof}
Since $g$ is slowly varying, for any $u > 0$ we have after a change of variables 
\begin{equation} \label{eq:eq1}
\frac{f(xu) - f(x) }{x f'(x)} = \int_1^u \frac{g(xy)}{g(x)} y^{-1} \d y \to \log u,
\end{equation}
where we used the uniform convergence theorem \cite[Theorem 1.2.1]{BGT}.
Assuming for a moment that the interchangeability of the limit and the 
integral is justified, we have
\[
\int_{\R} [ f(x + x z) - f(x) - f'(x) x z I(|z| \leq 1)] \nu_U (\d z) 
\sim x f'(x)  \int_{\R} [ \log |1 + z| - z  I(|z| \leq 1) ]  \nu_U (\d z).
\]
Thus, using also that $x^2 f''(x) \sim - x f'(x)$,
\[
\mathcal{A} f(x) \sim 
x f'(x) \left[ \gamma_U - \frac{\sigma_U^2}{2} +
\int_{\R} \left[ \log |1+z| - z I(|z| \leq 1) \right] \nu_U(\d z)
+ \int_{-1}^1 z \nu'(\d z) \right].
\]
Since $\varphi(f(x)) = x f'(x)$,  the statement follows.

So we only have to find a integrable majorant around infinity, around $0$, where the 
measure may be infinity, and around $-1$, where $\log |1+z|$ has a singularity. 

At infinity: Using the monotonicity of $g$, and 
$g(xu) \leq k_g g(x) g(u)$, $u \geq 1$, from (\ref{eq:eq1}) we obtain
\[
\frac{f(xu) - f(x) }{x f'(x)} \leq 
k_g g(u) \log u,
\]
which is integrable.

At $0$: By the mean value theorem we have
$f(x(1+z)) - f(x) = xz f'(\xi)$ with $\xi$ between $x$ and $x(1+z)$, and
$| f'(\xi) - f'(x) | = |(\xi - x) f''(\xi')| \leq |xz f''(\xi')|$, with
$\xi'$ between $x$ and $x(1+z)$. Therefore 
\[
\left| \frac{f (x(1+z)) - f(x)}{x f'(x)} - z \right|
= \left| z \left(\frac{f'(\xi)}{f'(x)} - 1 \right) \right|
\leq z^2 \frac{|x f''(\xi')|}{f'(x)},
\]
and since $x f''(x) \sim - f'(x)$, and $x f'(x)$ is slowly varying
$|x f''(\xi')| / f'(x)$ is uniformly bounded for $z \in [-1/2, 1/2]$.

At $-1$: Using the Potter bounds \cite[Theorem 1.5.6]{BGT}, for any $\varepsilon > 0$ 
there is a  
$c= c(\varepsilon) > 0$, such that
$g(xy)/ g(x) \leq 2 y^{-\varepsilon}$, $c/|x| \leq y \leq 1$. Thus for 
$c/|x| \leq |1+ z| \leq 1$ by (\ref{eq:eq1})
\[
\left| \frac{f(x(1+z)) - f(x)}{x f'(x)}\right|
= \left| \int_{|1+z|}^1 \frac{g(xy)}{g(x)} y^{-1} \d y \right|
\leq 2 \int_{|1+z|}^1 y^{-1-\varepsilon} \d y
\leq \frac{2}{\varepsilon} |1+z|^{-\varepsilon},
\]
which is integrable for some $\varepsilon$ with respect to $\nu_U$, according to the 
assumptions. Finally,
\[
\left| 
\int_{-1-c/|x|}^{-1+c/|x|}
\left[ 
\frac{f(x(1+z)) - f(x)}{g(x)} - z I(|z| \leq 1) \right] \nu_U(\d z) \right| \leq
\left( \frac{f(x)}{g(x)} + 1 \right) \nu_U(-1-c/|x|, -1+c/|x|), 
\]
and since $f(x)/g(x)$ is slowly varying the latter bound tends to 0 due to the 
integrability assumption.
\end{proof}

\begin{proof}[Proof of Proposition \ref{prop:genCDsubexp}]
Let $f(x) = (\log |x|)^\alpha$ for $|x| \geq 3$, and consider a smooth 
extension of it to $[-3,3]$, which is greater than, or equal to $1$.
We show that $f$ satisfies the conditions of Proposition \ref{prop:CDsubexp-gen}, except 
the integrability condition at $-1$.
Since $f$ is concave for $|x| \geq 3$ so it 
satisfies the assumptions of Lemma \ref{lemma:integral}. Moreover, 
\[
g(x) = x f'(x) = \alpha (\log |x|)^{\alpha -1} 
\]
is increasing and slowly varying, and $\log g(e^x)$ is concave, which, combined with 
Lemma \ref{lemma:concave} implies that $g(ux) \leq k_g g(u) g(x)$, $u,x \geq 1$, for some 
$k_g > 0$. Simple calculation shows that $x f''(x) \sim - f'(x)$. 
The function $\varphi(y) = \alpha y^{1-1/\alpha}$ is concave, and
$\varphi(f(x)) = g(x)$. Finally, as 
$g(x) \log x = \alpha (\log x)^\alpha$ the integrability condition 
at infinity is also satisfied.

Therefore, we only have to show that the integrability condition at $-1$ can be relaxed.
Let $\delta > 0$ be so small that $(1 - y)^\alpha \geq 1 - 2 \alpha y$ for 
$y \in [0,\delta]$ ($\delta(\alpha) = 1 - 2^{-1/(1-\alpha)}$ works). Then, 
with $u = |1 + z| \in [|x|^{-\delta}, 1]$
\[
\left| \frac{f(xu) - f(x)}{x f'(x)} \right|  =
\alpha^{-1} \log |x| \left[ 1 - \left( 1 - \frac{\log u^{-1}}{\log |x|} \right)^\alpha 
\right] \leq 2 \log u^{-1}.
\]
While, for $u \leq |x|^{-\delta}$
\[
\begin{split}
& \int_{|1+z| \leq |x|^{-\delta} } 
\left| \frac{f(x(1+z)) - f(x)}{x f'(x)} - z I(|z| \leq 1)
\right| \nu_U(\d z) \\
& \leq 
(\alpha^{-1} \log |x| + 1) \nu_U ((-1-|x|^{-\delta}, -1+|x|^{-\delta})) \to 0.
\end{split}
\]
At the last step we used that the integrability condition 
$\int_{-3/2}^{-1/2} | \log | 1 + z | \, | \nu_U(\d z) < \infty$ implies
$\log |x| \, \nu_U((-1-1/|x|, -1+1/|x|)) \to 0$.
\end{proof}

\begin{proof}[Proof of Proposition \ref{prop:genCDsubexp2}]
Let $f(x)  = \exp\{ \gamma (\log |x|)^\alpha\}$, $\alpha \in (0,1)$, $\gamma > 0$, 
for  $|x| \geq e$, and consider a smooth extension of it to $[-e,e]$, which is greater 
than, or equal to $1$. First, we show that the function $f$ satisfies 
the conditions of Proposition \ref{prop:CDsubexp-gen}, except the integrability condition 
at infinity, and at $-1$.

Simply,
\[
g(x) = x f'(x) = \gamma \alpha (\log |x|)^{\alpha -1} \exp \{ \gamma (\log |x|)^\alpha \},
\]
which is an increasing, slowly varying function on $(e,\infty)$. Moreover, $\log g(e^x)$ 
is concave, for $x \geq 1$, therefore Lemma \ref{lemma:concave} implies that 
$g(ux) \leq k_g g(u) g(x)$ for some 
$k_g>0$. Straightforward calculation shows that  $x f''(x) \sim - f'(x)$. Finally, 
$\varphi(f(x) ) = g(x)$ for the concave function 
$\varphi(x) = \alpha \gamma^{1/\alpha} x \, (\log x)^{1-1/\alpha}$.

Now we prove that the integrability condition at infinity and at $-1$ can be relaxed. We 
start with the condition at infinity. Let $u = |z+1| \geq 1$, and write
\[
I(u) =  \frac{f(xu) - f(x)}{x f'(x)}  =
\frac{1}{\gamma \alpha} (\log |x|)^{1-\alpha} \left[ \exp \left\{ 
\gamma (\log |x|u)^\alpha - \gamma (\log |x|)^\alpha \right\} -1 \right].
\]
For $\log u \leq \log |x|$, using $(1 + y)^\alpha \leq 1 + \alpha y$, $y \geq 0$, we have
\[
(\log |x|u)^\alpha - (\log |x|)^\alpha =
(\log |x|)^\alpha \left[ \left( 1 + \frac{\log u}{\log |x|} \right)^\alpha - 1 \right] 
\leq \alpha (\log |x|)^{\alpha -1 } \log u.
\]
If $\gamma \alpha \log u \, (\log |x|)^{\alpha -1} \leq 1$, then using $e^y - 1 \leq 2y$ 
for $y \in [0,1]$ we have
\[
I(u) \leq 2 \log u,
\]
which is integrable. While, for $\gamma \alpha \log u \, (\log |x|)^{\alpha -1} \geq 1$,
noting that $ (\log |x|)^{\alpha -1 } \log u  \leq ( \log u)^\alpha$, we obtain 
\[
I(u) \leq \frac{1}{\gamma \alpha} (\log |x|)^{1 - \alpha} 
e^{\gamma \alpha (\log u)^\alpha }
\leq \log u \, e^{\gamma \alpha (\log u)^\alpha } \leq c_1 e^{\gamma (\log u)^\alpha },
\]
with $c_1 >0$, which is again integrable according to our assumptions. Here, $c_1, c_2, 
\ldots $ are strictly positive constants, whose value are not important.

For $\log u \geq \log |x|$ we use the inequality
$(1+y)^\alpha - y^\alpha \leq 1 - y^\alpha/2$, which holds for $y \in (0,\delta)$, for 
some $\delta > 0$ ($\delta(\alpha) = 1/(2^{1/(1-\alpha)}-1)$ works). 
If $\log |x| \leq \delta \log u$, then
\[
I(u) \leq \frac{1}{\gamma \alpha} (\log |x|)^{1-\alpha} 
\exp \left\{ \gamma (\log u)^\alpha - \frac{\gamma}{2} (\log |x|)^\alpha \right\}
\leq c_2  \exp \{ \gamma (\log u)^\alpha \},
\]
with some $c_2 > 0$, which is integrable. Otherwise, for  $\log |x| \geq \delta \log u$
\[
\begin{split}
I(u) & = 
\frac{1}{\gamma \alpha} (\log |x|)^{1-\alpha} \exp \left\{ 
\gamma (\log u)^\alpha \left[ \left( 1  + \frac{\log |x|}{\log u} \right)^\alpha -
 \left( \frac{\log |x|}{\log u} \right)^\alpha \right] \right\} \\
& \leq \frac{1}{\gamma \alpha} (\log |x|)^{1-\alpha} \exp \left\{ 
\gamma (\log u)^\alpha \eta \right\} \\
& \leq c_3 \exp \{ \gamma (\log u)^\alpha \},
\end{split}
\]
with
$\eta = \sup_{y \in [\delta, 1]} [ (1 + y)^\alpha - y^\alpha ] < 1$,
and some $c_3 > 0$.
\smallskip

We turn to the integrability condition at $-1$. Now $u = |1 + z| \in [0,1]$.
The integral condition $\int_{-3/2}^{-1/2} |\log |1+z| | \, \nu_U(\d z) < \infty$ implies
$(\log |x|)^{1 - \alpha} \nu_U(( -1-1/|x|, -1+1/|x|)) \to 0$, and so for any 
$\varepsilon > 0$ it also holds that
$(\log |x|)^{1- \alpha} \nu_U(-1- |x|^{-\varepsilon}, -1 + |x|^{-\varepsilon}) \to 0$. 
Therefore, we 
may and do assume that $|x|^{-\varepsilon} \leq u  \leq 1$. Then $\log u / \log |x| \in 
[-\varepsilon, 0]$. Let $\varepsilon > 0$ be small enough, such that 
$(1 - y)^\alpha  \geq 1 - 2 \alpha y $, for $y \in [0,\varepsilon]$
($\varepsilon(\alpha) = 1 - 2^{-1/(1-\alpha)}$ works). Then
\[
\begin{split}
I(u) & =  
\frac{1}{\gamma \alpha} (\log |x|)^{1-\alpha} \left[1 - \exp \left\{ 
- \gamma (\log |x|)^\alpha \left[ 1- \left( 1 + \frac{\log u}{\log |x|} \right)^\alpha 
\right] \right\} \right] \\
& \leq \frac{1}{\gamma \alpha} (\log |x|)^{1-\alpha} \left[1 - \exp \left\{ 
 2 \gamma \alpha (\log |x|)^{\alpha-1} \log u \right\} \right]  \\
& \leq - 2 \log u,
\end{split}
\]
where at the last step we used the simple inequality $1 - e^{u} \leq - u$. Since
$\log |1+z|$ is integrable around $-1$, the statement is proved.
\end{proof}

\begin{proof}[Proof of Proposition \ref{prop:genCD3}]
Let $f(x) = |x|^\beta$ for $|x| \geq 1$, and consider an even, smooth nonnegative 
extension of it to $[-1,1]$. Since $x^\beta$ is concave
$f$ satisfies the assumptions of Lemma \ref{lemma:integral}. For $\beta < 1$ we have 
$\lim_{|x| \to \infty} f'(x) = 0$, while $f'(x)$ is bounded for $\beta =1$.

We have
\[
\int_{-1+x^{-1}}^1 [ f(x + x z) - f(x) - f'(x) x z] \nu_U (\d z)
= x^\beta \int_{-1+x^{-1}}^1 [ (1 + z)^\beta - 1 - \beta z] \nu_U (\d z),
\]
and
\[
\left| \int_{-1}^{-1+x^{-1}} [ f(x + x z) - f(x) - f'(x) x z] \nu_U (\d z) 
\right|
\leq 3 x^\beta  \nu_U ([-1, -1+x^{-1}]) = o(x^{\beta}),
\]
since $\nu_U(\{-1\}) =0$. Therefore
\[
\int_{-1}^1 [ f(x + x z) - f(x) - f'(x) x z] \nu_U (\d z)
= x^\beta \int_{-1}^1 [ (1 + z)^\beta - 1 - \beta z] \nu_U (\d z)
+ o(x^{\beta}).
\]
Using the same  argument as above (now the exceptional set is $[-1-x^{-1}, -1]$), 
we obtain
\[
\int_{|z|>1} [f(x+xz ) - f(x) ] \nu_U (\d z) 
= x^\beta \int_{|z| >1} [ |1 + z|^\beta - 1 ]  \nu_U (\d z) + o(x^\beta).
\]

Summarizing, we obtain
\begin{equation} \label{eq:int-decomp2}
\begin{split}
&\dint_{\R^2} [ f(x+x z_1 +z_2) - f(x) - f'(x) (x z_1 +  z_2) I(|z| \leq 1) ]
\nu_{UL}(\d z) \\
& = x^\beta  \left[
\int_{\R}  [ |1 + z|^\beta - 1 - z \beta I(|z| \leq 1) ] \nu_U (\d z) + 
\beta \int_{-1}^1 z \nu'(\d z) \right] + o(|x|^\beta).
\end{split}
\end{equation}
Now the statement follows. According to (\ref{eq:int-decomp2}), for $x > 0$ large enough
\[
\mathcal{A}_n f(x) \sim
\beta x^\beta \left[ \gamma_U - \frac{\sigma_U^2 (1-\beta)}{2}  
+ \int_{\R} \frac{ |1 + z|^\beta - 1 - z \beta I(|z| \leq 1) }{\beta} \nu_U (\d z) + 
\int_{-1}^1 z \nu'(\d z)  \right].
\]
Since the expression in the square bracket is negative, we obtain (\ref{eq:CD3}).
\end{proof}

Next, we handle the diffusion case. Here the result is an easy consequence of general 
results from \cite{K}.

%

\begin{proof}[Proof of Proposition \ref{prop:trans}]
It is enough to prove the second part.
We use Lemma 3.10 in \cite[p.~94]{K}. Let us choose $\alpha > 0$ small enough,
such that
\[
\gamma_U > \frac{\sigma_U^2 (1 + \alpha)}{2}.
\]
This is possible due to the assumptions. Let $f(x) = k - |x|^{-\alpha}$ outside of a
neighborhood of 0, where $k > 0$ is chosen later.
Then for $x > 0$
\[
\begin{split}
\mathcal{A} f(x) & =  (\gamma_U x + \gamma_L) \alpha x^{-\alpha -1} -
\frac{1}{2} ( x^2 \sigma_U^2 + 2 x \sigma_{UL} + \sigma_L^2) \alpha (\alpha+1)
x^{-\alpha-2} \\
& =
\alpha x^{-\alpha} \left( \gamma_U - \frac{(\alpha + 1) \sigma_U^2}{2} \right)
+ O(x^{-\alpha -1}).
\end{split}
\]
The same calculation for negative $x$ shows that $\mathcal{A} f(x) \geq 0$ for
$| x | \geq x_0 > 0$. Let $k = x_0^{-\alpha}$. Clearly,
$\sup_{|x| \geq x_0} f(x) \leq k$, and $f( \pm x_0) = 0$.
From \cite[Lemma 3.10]{K} the nonrecurrence follows relative to the domain 
$( -x_0, x_0)$. From \cite[Lemma 4.1]{K} the nonrecurrence relative to 
any domain follows.
\end{proof}

\subsection{Petite sets} \label{subsect:proof-petite}

\begin{proof}[Proof of Proposition \ref{prop:petite-2}]
Since $\lim_{n \to \infty} \p_x \{ V_n \in A \} = \pi (A)$ for any $x \in \R$, we see 
that the skeleton process $(V_n)_{n \in \N}$ is irreducible with respect to the invariant 
measure $\pi$. According to our assumptions the interior of the support is nonempty, 
therefore Theorem 3.4 in \cite{MT1} implies that the compact sets are petite sets.
\end{proof}

\subsection{Proof of the theorems}

\begin{proof}[Proof of Theorem \ref{thm:genCD2}]
The process is clearly nonexplosive, and
Propositions \ref{prop:genCD2} shows that the ergodicity condition holds. 
Thus Theorem 5.1 \cite{MT3} proves the statement.
\end{proof}

\begin{proof}[Proof of Theorem \ref{thm:gensubexp}]
Here we use the notation in \cite{DFG}.
A petite set for the skeleton chain 
is petite set for the continuous process. In particular, the compact set in condition 
(\ref{eq:CDsubexp}) is petite. Proposition \ref{prop:genCDsubexp} and 
Theorem 3.4 \cite{DFG} imply that the assumptions of Theorem 3.2 \cite{DFG} are satisfied 
with $f(x) = (\log |x|)^\alpha$ and $\varphi(x) = x^{1 - 1/\alpha}$. From the discussion 
after Theorem 3.2 \cite{DFG} we see that for the rate of convergence corresponding to the 
total variation distance we have 
\[
\| \p_x \{ V_t \in \cdot \} - \pi \| \leq  C  (\log |x|)^\alpha r_*(t)^{-1} ,
\]
with $r_*(t) = \varphi(H_\varphi^{\leftarrow}(t))$, where 
$H_\varphi(t) = \int_1^t \varphi(s)^{-1} \d s$, and $H_\varphi^{\leftarrow}$ is the 
inverse function of $H_\varphi$. After a short calculation we see that this is exactly the 
statement.
\end{proof}

\begin{proof}[Proof of Theorem \ref{thm:gensubexp2}] 
The proof is the same as the previous one. Now
$f(x) = \exp \{ \gamma (\log |x|)^\alpha \}$ and 
$\varphi(x) = x (\log x)^{1 - 1/\alpha}$. Short calculation gives that $H_\varphi(t) = 
\alpha (\log t)^{1/\alpha}$, and the statement follows.
\end{proof}

\begin{proof}[Proof of Theorem \ref{thm:genCD3}]
Proposition \ref{prop:genCD3} shows that the 
exponential ergodicity condition holds, thus Theorem 6.1 \cite{MT3} implies the statement.
\end{proof}

\bibliographystyle{plain}
\bibliography{GOUergod}

\end{document}